\documentclass[12pt,a4paper]{article}
\usepackage{amsmath,amssymb,amsfonts, amsthm}
\def\Bbb{\mathbb}

\title{\bf On the Asymptotic Behavior of Dedekind Sums}

\author{Kurt Girstmair\\
Institut f\"ur Mathematik \\
Universit\"at Innsbruck   \\
Technikerstr. 13/7        \\
A-6020 Innsbruck, Austria \\
Kurt.Girstmair@uibk.ac.at}
\date{}

\makeatletter
\let\@@maketitle=\maketitle
\def\maketitle{\def\thispagestyle##1{\relax}\@@maketitle}
\makeatother
\newtheorem{theorem}{Theorem}

\newenvironment{rem}{{\em Remark.}}{}
\newenvironment{exmp}{{\em Example.}}{}

\def\BE{\begin{equation}}
\def\EE{\end{equation}}
\def\BD{\begin{displaymath}}
\def\ED{\end{displaymath}}
\def\BA{\begin{array}}
\def\EA{\end{array}}
\def\BEA{\begin{eqnarray*}}
\def\EEA{\end{eqnarray*}}
\def\BI{\bibitem}

\def\Z{\Bbb Z}

\def\R{\Bbb R}

\def\phi{\varphi}

\def\MB{\mbox}
\def\LD{\ldots}
\def\OV{\overline}

\def\DED{Dedekind }

\def\FL{$\lfloor\enspace\rfloor$}
\def\CE{$\lceil\enspace\rceil$}

\begin{document}
\maketitle

\begin{abstract}
\noindent
Let $z$ be a real quadratic irrational. We compare the
asymptotic behavior of Dedekind sums $S(p_k,q_k)$ belonging to convergents $p_k/q_k$ of the {\em regular}
continued fraction expansion of $z$ with that of Dedekind sums $S(s_j/t_j)$ belonging to convergents $s_j/t_j$
of the {\em negative regular} continued fraction expansion of $z$. Whereas the three main cases of this behavior
are closely related, a more detailed study of the most interesting case (in which the Dedekind sums remain bounded)
exhibits some marked differences, since the cluster points depend on the respective periods of these expansions.
We show in which cases cluster points of $S(s_j,t_j)$ can coincide with cluster points of $S(p_k,q_k)$.
An important tool for our purpose is a
criterion that says which convergents $s_j/t_j$ of $z$ are convergents $p_k/q_k$.
\end{abstract}

Continued fraction expansions of quadratic irrationals

\noindent
{\em AMS Subject Classification:}
\\ {\em Primary:} 11 F 20,
\\ {\em Secondary:} 11 A 55

\newpage

\section*{1. Introduction and results}

Let $z$ be a real irrational number. We consider the {\em regular} continued fraction expansion
\BD
 z=a_0+\frac{1}{\displaystyle a_1+{\frac{1}{\displaystyle a_2+\LD}}}=\lfloor a_0,a_1,a_2,\LD\rfloor,
\ED
which is defined by the following algorithm:
\BD
  z_0=z,\ a_0=\lfloor z_0\rfloor,\ z_{j+1}=1/(z_j-a_j),\ a_{j+1}=\lfloor z_{j+1}\rfloor,\ j\ge 0.
\ED
The {\em convergents} $p_0/q_0$, $p_1/q_1$, $p_2/q_2$, \LD\, of this expansion are given by
\BD
   p_{-1}=1, q_{-1}=0,p_{0}=a_0,q_{0}=1, p_{k+1}=a_{k+1}p_k+p_{k-1}, q_{k+1}=a_{k+1}q_k+q_{k-1}, k\ge 0.
\ED
For the sake of simplicity we call these convergents the $\lfloor\enspace\rfloor$-{\em convergents} of $z$.

We will compare the regular continued fraction expansion of $z$ with the {\em negative regular} continued
fraction expansion
\BD
 z=c_0-\frac{1}{\displaystyle c_1-\frac{1}{\displaystyle c_2-\LD}}=\lceil c_0,c_1,c_2,\LD\rceil,
\ED
which is defined by the algorithm
\BD
z_0=z,\ c_0=\lceil z_0\rceil,\ z_{j+1}=1/(c_j-z_j),\ c_{j+1}=\lceil z_{j+1}\rceil,\ j\ge 0.
\ED
The {\em convergents} $s_0/t_0$, $s_1/t_1$, $s_2/t_2$, \LD\, of this expansion are given by
\BD
   s_{-1}=1, t_{-1}=0,s_{0}=c_0,t_{0}=1, s_{j+1}=c_{j+1}s_j-s_{j-1}, t_{j+1}=c_{j+1}t_j-t_{j-1}, j\ge 0.
\ED
They are called the $\lceil\enspace\rceil$-{\em convergents} of $z$.
Note that $c_j\ge 2$ for all $j\ge 1$, and $c_j\ge 3$ for infinitely many indices $j\ge 0$ (see \cite{Zu}).
Henceforth we call this continued fraction expansion simply the {\em negative regular expansion}.
This expansion has aroused some interest due to the work of Hirzebruch and Zagier
about class numbers of quadratic fields (see \cite{Za}, \cite[p. 136]{Za2}).

It is well-known that \FL-convergents of $z$ have optimal approximation properties which \CE-convergents have not,
unless they happen to be \FL-convergents (see \cite[p. 44ff.]{Pe}).
In general, \CE-convergents are only intercalary fractions of \FL-convergents: If
\BD
\frac{p_k}{q_k}\ \MB{ and }\ \frac{p_{k+1}}{q_{k+1}}=\frac{a_{k+1}p_k+p_{k-1}}{a_{k+1}q_k+q_{k-1}}
\ED
are two adjacent \FL-convergents of $z$, then the intercalary fractions are
\BD
  \frac{p_k+p_{k-1}}{q_k+q_{k-1}}, \frac{2p_k+p_{k-1}}{2q_k+q_{k-1}},\LD,
  \frac{(a_{k+1}-1)p_k+p_{k-1}}{(a_{k+1}-1)q_k+q_{k-1}}.
\ED
They also approximate $z$ successively, but not as well as \FL-convergents do.

It is, therefore, desirable to be able to decide whether a \CE-convergent is a \FL-convergent. We give the following criterion.

\begin{theorem} 
\label{t1}
Let $z=\lceil c_0,c_1,c_2, \LD\rceil$ be an irrational number and $s_j/t_j$, $j\ge 1$,
a $\lceil\enspace\rceil$-convergent of $z$.
Then $s_j/t_j$ is a $\lfloor\enspace\rfloor$-convergent of $z$ iff $c_{j+1}\ge 3$.
In particular, the sequence $s_j/t_j$, $j\ge 1$, contains infinitely many
$\lfloor\enspace\rfloor$-convergents of $z$.
\end{theorem} 

\begin{exmp}
Let $z=e=2.71828\LD$ be Euler's number, whose regular continued fraction expansion is
\BD
 \lfloor 2,\OV{1,2k,1}\rfloor_{k=1}^{\infty}=\lfloor 2,1,2,1,1,4,1,1,6,1,\LD\rfloor,
\ED
where we have used Perron's notation (see \cite[p. 124]{Pe}). By means of the transition formula
\BE
\label{1.0}
  \lfloor a_0,a_1,a_2,a_3,a_4,\LD\rfloor=\lceil a_0+1,2^{(a_1-1)},a_2+2,2^{(a_3-1)},a_4+2,\LD\rceil
\EE
(see \cite{Bo}, p. 93) one easily obtains
\BD
 e=\lceil \OV{3,4k,3,2^{(4k-1)}}\rceil_{k=1}^{\infty}=\lceil 3,4,3,2^{(3)},3,8,3, 2^{(7)},\LD\rceil;
\ED
here $2^{(j)}$ stands for a sequence of $j$ numbers $2$. According to Theorem \ref{t1}, the \CE-convergents
\BD
  \frac{s_1}{t_1}=\frac{11}{4}, \frac{s_5}{t_5}=\frac{87}{32}, \frac{s_6}{t_6}=\frac{193}{71},
  \frac{s_7}{t_7}=\frac{1457}{536}, \frac{s_{15}}{t_{15}}=\frac{23225}{8544}, \frac{s_{16}}{t_{16}}=\frac{49171}{18089},
  \frac{s_{17}}{t_{17}}=\frac{566827}{208524}
\ED
are \FL-convergents of $e$. Indeed, if the latter are denoted by $p_k/q_k$, $k\ge 0$,
we find that the above \CE-convergents coincide with $p_3/q_3$, $p_5/q_5$, \LD, $p_{15}/q_{15}$, respectively.
As intercalary fractions we have, for instance, $s_2/t_2=(p_4+p_3)/(q_4+q_3)$ or $s_8/t_8=(p_{10}+p_9)/(q_{10}+q_9)$.
\end{exmp}

The main application of Theorem \ref{t1} in this paper is a comparative study of the asymptotic behavior of
\DED sums with arguments near quadratic irrationals.
For an integer $a$ and a natural number $b$ let
\BD
  s(a,b)=\sum_{k=1}^b((k/b))((ak/b))
\ED
be the classical \DED sum, where
\BD
 ((x))=\left\{\begin{array}{ll}
                 x-\lfloor x\rfloor-1/2 & \MB{ if } x\in\R\smallsetminus \Z; \\
                 0 & \MB{ if } x\in \Z
               \end{array}\right.
\ED
(see, for instance, \cite[p. 1]{RaGr}). We work with $S(a,b)=12s(a,b)$ instead of $s(a,b)$
and compare the asymptotic behavior of the \DED sums $S(p_k,q_k)$ of a quadratic irrational $z$ with
the asymptotic behavior of $S(s_j,t_j)$ when $k$ and $j$ tend to infinity. Theorem \ref{t1} has the effect
that the asymptotic behavior is, roughly speaking, the same in both cases.

For a quadratic irrational $z$ both expansions are periodic, i.e.,
\BD
 z=\lfloor a_0,a_1,\LD a_q,\OV{b_1,b_2,\LD,b_l}\rfloor=\lceil c_0,c_1,\LD c_r,\OV{d_1,\LD,d_m}\rceil,
\ED
where $(b_1,\LD,b_l)$ and $(d_1,\LD,d_m)$ are the respective periods (see \cite[Satz 14, Satz 15]{Zu}).
Here $l$ and $m$ are smallest possible.
In addition, $q$ and $r$ are smallest possible for this choice of $l$ and $m$.
In the purely periodic case, we put $q=-1$ and $r=-1$.
Further, let
\BD
   L=\left\{\BA{lr} l & \MB{ if } $l$ \MB{ is even;}\\
                    2l & \MB{ if } $l$ \MB{ is odd.}
             \EA\right.
\ED
Then $(b_1,\LD,b_L)$ is also a period of the regular continued fraction expansion of $z$.
If $l$ is odd, it has the form $(b_1,\LD,b_l,b_1,\LD,b_l)$.
The three cases of Theorem \ref{t2} (below) are connected with the values of
\BD
  B=(-1)^q\sum_{j=1}^L(-1)^{j-1}b_j\ \MB{ and }\ D=\frac 1m\sum_{j=1}^m d_j.
\ED
It should be noted that $B=0$ if $l$ is odd or if the period $(b_1,\LD,b_l)$ is symmetric.

\begin{theorem} 
\label{t2}

In the above setting,  $B>0$, $B=0$, $B<0$ iff  $D<3$, $D=3$, $D>3$, respectively.
\\ Suppose that $j$ and $k$ tend to $\infty$.
Then
\\ both $S(p_k,q_k)$ and $S(s_j,t_j)$ tend to $\infty$ if $B>0$ $($or $D<3)$,
\\ these quantities remain bounded if $B=0$ $($or $D=3)$,
\\ and they tend to $-\infty$ if $B<0$ $($or $D>3)$.

\end{theorem} 

In the remainder of this section we assume $B=0$, i.e., $D=3$. In the paper \cite{Gi} it has been shown that
the \DED sums $S(p_k,q_k)$
accumulate near $L$ cluster points $U_h$, $h=1,\LD,L$. They are given by
\BE
\label{1.1}
  U_h=A+(-1)^q\sum_{k=1}^h(-1)^{k-1}b_k + \left\{\begin{array}{ll}
                                             z+1/u_h-3 & \MB{ if } q+h\MB{ is odd };  \\
                                             z-1/u_h   & \MB{ if } q+h\MB{ is even}.
                                           \end{array}\right.
\EE
Here $A=\sum_{k=0}^q(-1)^{k-1}a_k$ and $u_h=\lfloor \OV{b_h,b_{h-1},\LD,b_1,b_L,b_{L-1}, \LD,b_{h+1}}\rfloor$ is a
purely periodic quadratic irrational, $h=1,\LD,L$. On the other hand, we shall see that $S(s_j,t_j)$ accumulate
near cluster points $V_i$, $i=1,\LD,m$. They are given by
\BE
\label{1.3}
  V_i=C+\sum_{j=1}^i(3-d_j)+z-1/v_i-3,
\EE
where $C=\sum_{j=0}^r(3-c_j)$ and $v_i=\lceil\OV{d_i,d_{i-1},\LD,d_1,d_m,d_{m-1},\LD,d_{i+1}}\rceil$
is also a quadratic irrational, $i=1,\LD,m$. In the purely periodic cases, we have $A=0$ and $C=0$.

\begin{theorem} 
\label{t3}

In the above setting, the cluster points $U_h$, $h=1,\LD,L$, are pairwise distinct, as are the cluster points
$V_i,i=1,\LD,m$.
\end{theorem} 

The following theorem concerns cluster points $V_i$ which coincide
with cluster points $U_h$.

\begin{theorem} 
\label{t4}
Let $(d_1,\LD,d_m)$ be the period of the negative regular expansion of $z$. Let $i\in\{1,\LD,m\}$ be such that
$d_{i+1}\ge 3$ $($if $i=m$ this means $d_1\ge 3$$)$. Then the cluster point $V_i$ coincides with some cluster point
$U_h$ such that $q+h$ is odd. A cluster point $U_h$ with $q+h$ even cannot coincide with a cluster point $V_i$.
\end{theorem} 

\begin{rem}
One can show that the period obtained from $(b_1,\LD,b_L)$ by means of the transition formula (\ref{1.0})
is shortest possible among the possible periods of the negative regular expansion of $z$
--- but we abstain from doing this here. Accordingly, there are
$L/2$ indices $i\in\{1,\LD,m\}$ such that $d_{i+1}\ge 3$, the corresponding values of $d_{i+1}$
being either $b_1+2,b_3+2, \LD, b_{L-1}+2$ or $b_2+2, b_4+2,\LD,b_L+2$.
Hence there are $L/2$ coincidences of cluster points $V_i$ with cluster points $U_h$.
This means that {\em each} cluster point $U_h$, $q+h$ odd, coincides with such a $V_i$. So we have a fairly
complete picture: The cluster points $U_h$, $q+h$ odd, coincide with the cluster points $V_i$, $d_{i+1}\ge 3$.
Moreover, there are no coincidences of cluster points $U_h$, $q+h$ even, and cluster points $V_i$.
\end{rem}

\begin{exmp}
Let $z=1/\sqrt{53}=\lfloor 0,7,\OV{3,1,1,3,14}\rfloor=
\lceil 1, 2^{(6)},\OV{5,3,2,2,16,2,2,3,5,2^{(13)}}\rceil$. Since $l=5$ is odd,
the sequences $S(p_k,q_k)$ and $S(s_j,t_j)$ remain bounded. Because we have five ($=L/2$) entries $\ge 3$
in the period of the negative regular expansion of $z$, we have five common cluster points, namely
$V_1=U_2=(636+60\sqrt{53})/371=2.89166\LD$, $V_4=U_4=(-159+54\sqrt{53})/53=4.41747\LD$, $V_7=U_6
=(-2862+60\sqrt{53})/371=-6.53690\LD$, $V_8=U_8=(-1749+57\sqrt{53})/212=-6.29261\LD$,
and $V_{22}=U_{10}=(477+57\sqrt{53})/212=4.20738\LD$ As we expect, there are no further coincidences between
the remaining 17 cluster point $V_i$ and the remaining 5 cluster points $U_h$, although these points may lie close
together, like $V_{12}=(-13833+97\sqrt{53})/2332=-5.62900\LD$ and $U_7=(-1749-46\sqrt{53})/371=-5.61694\LD$
\end{exmp}

\section*{2. Proofs}

\begin{proof}[Proof of Theorem \ref{t1}]
Let $z_{j+1}$ denote the $(j+1)$th complete quotient of the negative regular
expansion of $z$.
Accordingly,
\BD
  z=\frac{z_{j+1}s_j-s_{j-1}}{z_{j+1}t_j-t_{j-1}}.
\ED
This gives
\BD
  z-\frac{s_j}{t_j}=\frac{s_jt_{j-1}-s_{j-1}t_j}{t_j(z_{j+1}t_j-t_{j-1})}.
\ED
Here we use $s_jt_{j-1}-s_{j-1}t_j=-1$ (see \cite[formula (4)]{Zu}) and obtain
\BE
\label{2.0.1}
  \left|z-\frac{s_j}{t_j}\right|=\frac{1}{t_j(z_{j+1}t_j-t_{j-1})}=\frac{1}{t_j^2}\cdot\frac{t_j}{z_{j+1}t_j-t_{j-1}}.
\EE
We need the regular continued fraction expansion $s_j/t_j=\lfloor a_0,\LD,a_n\rfloor$. However, we do not
require $a_n\ge 2$; hence
we may assume that $n$ is odd. Let $p_0/q_0$,
\LD, $p_n/q_n$ be the convergents of this expansion, in particular, $s_j/t_j=p_n/q_n$.
Now we can apply a criterion of Legendre
(see \cite[p. 39]{Pe}) to the identity (\ref{2.0.1}); thereby, we see that $s_j/t_j$ is a
$\lfloor\enspace\rfloor$-convergent of $z$, iff
\BE
\label{2.0}
    \frac{t_j}{z_{j+1}t_j-t_{j-1}}<\frac{q_n}{q_n+q_{n-1}}.
\EE
Let $s^*\in\{1,\LD,t_j-1\}$ denote the inverse of $s_j$ mod $t_j$, i.e., $s_js^*\equiv 1$ mod $t_j$ (observe $j\ge 1$).
Since $s_jt_{j-1}-s_{j-1}t_j=-1$, we see that $t_{j-1}\equiv -s^*$ mod $t_j$, and $1\le t_{j-1}<t_j$ implies
$t_{j-1}=t_j-s^*$. On the other hand, $s_j=p_n$ and $t_j=q_n$, which yields
\BD
  s_jq_{n-1}-p_{n-1}t_j=(-1)^{n-1}=1
\ED
(see \cite[p. 25]{Pe}).
Accordingly, $q_{n-1}=s^*$. Therefore, the condition $(\ref{2.0})$ is equivalent to
\BD
  z_{j+1}t_j-(t_j-s^*)>t_j+s^*,\ \MB{ i.e., to }\ z_{j+1}>2.
\ED
Recall that $z_{j+1}>2$ iff $c_{j+1}\ge 3$. Recall, further, that $c_{j+1}\ge 3$ holds for infinitely many
indices $j$. Thereby, we obtain the desired result.
\end{proof}

\begin{proof}[Proof of Theorem \ref{t2}]
Let $z=\lfloor a_0,a_1,\LD a_q,\OV{b_1,b_2,\LD,b_l}\rfloor=\lceil c_0,c_1,\LD c_r,\OV{d_1,\LD,d_m}\rceil$ be as
above.
In \cite{Gi} we studied the asymptotic behavior of $S(p_k,q_k)$ for the \FL-convergents $p_k/q_k$
of $z$.  Indeed, if $k=q+nL+h$, $n\ge 0$, $h\in\{1,\LD, L\}$,
\BE
\label{2.1}
  S(p_k,q_k)=A+nB+(-1)^q\sum_{j=1}^{h}(-1)^{j-1}b_j+
                                             \left\{\begin{array}{ll}
                                             (p_k+q_{k-1})/q_k-3 & \MB{ if } k\MB{ is odd };  \\
                                             (p_k-q_{k-1})/q_k   & \MB{ if } k\MB{ is even}.
                                           \end{array}\right.
\EE
Here $A$ and $B$ are as in Section 1. Since $p_k/q_k\to z$ for $k\to \infty$ and $0\le q_{k-1}/q_{k}\le 1$ for $k\ge 0$,
this gives
\begin{eqnarray}
\label{2.3}
  S(p_k,q_k)\to \infty \MB{ if } B>0,\nonumber\\
  S(p_k,q_k)\MB{ remains bounded} \MB{ if } B=0, \MB{ and }\\
  S(p_k,q_k) \to -\infty \MB{ if } B<0.\nonumber
\end{eqnarray}
A similar asymptotic behavior takes place for the \DED sums $S(s_j,t_j)$ belonging to the
\CE-convergents $s_j/t_j$ of $z$.
Indeed, a formula of Hirzebruch, Zagier and Myerson  (see \cite{My}) says
\BD
  S(s_j,t_j)=\sum_{k=0}^j(3-c_k)+(s_j-t_{j-1})/t_j-3,
\ED
where we have written $z=\lceil c_0,c_1,c_2,\LD\rfloor$ disregarding the period.
Hence we have, for $j=r+nm+i$, $n\ge 0$, $i\in\{1,\LD, m\}$,
\BE
\label{2.5}
  S(s_j,t_j)=C+nm(3-D)+\sum_{k=0}^{i}(3-c_k)+(s_j-t_{j-1})/t_j-3,
\EE
where $C$ and $D$ are as in Section 1.
Observe $s_j/t_j\to z$ for $j\to\infty$ and $0\le t_{j-1}/t_j\le 1$ for $j\ge 0$ (see \cite[Satz 4, Satz 1]{Zu}).
Then we obtain
\begin{eqnarray}
\label{2.7}
  S(s_j,t_j)\to \infty \MB{ if } D< 3,\nonumber\\
  S(s_j,t_j)\MB{ remains bounded} \MB{ if } D=3, \MB{ and }\\
  S(s_j,t_j) \to -\infty \MB{ if } D>3.\nonumber
\end{eqnarray}
By Theorem \ref{t1}, the sequence $s_j/t_j$ contains infinitely many
\FL-convergents $p_k/q_k$ of $z$. Therefore, the asymptotic behavior of $S(p_k,q_k)$
for these convergents of $z$ in the sense of (\ref{2.3}) must be the same as in the sense of (\ref{2.7}).
This, however, yields Theorem \ref{t2}, in particular, the connection between $B$ and $D$.
\end{proof}

\begin{rem} The relation between the quantities $B$ and $D$ can also be proved in a simple and direct way
by means of the transition formula (\ref{1.0}).
\end{rem}

\begin{proof}[Proof of Theorem \ref{t3}] In the paper \cite{Gi} we deduced formula (\ref{1.1}) from (\ref{2.1}). In
exactly the same
way one can deduce (\ref{1.3}) from (\ref{2.5}). First we show that the cluster points $V_i$, $i=1,\LD,m$, are pairwise
distinct.
Suppose that $V_i=V_j$ for some $i,j\in\{1,\LD,m\}$.  This implies
\BD
  1/v_i=1/v_j+a,\ \MB{ for some }a\in\Z.
\ED
Since $v_i, v_j$ are both $>1$,
we have $0<1/v_i,1/v_j<1$, whence $a=0$ and $v_i=v_j$ follows. Since $m$ is smallest possible,
$v_i$ and $v_j$ have different negative regular expansions if $i\ne j$. Therefore,  $i=j$.

In the case of the cluster points $U_h$, $h=1,\LD,L$, the proof is slightly more subtle.
Suppose, first, that $l$ is odd.
Then the values $1/u_h$, $h=1,3,\LD,l$, and $1/u_{l+h}=1/u_h$, $h=2,4,\LD,l-1$, appear with the same sign in
(\ref{1.1}), whereas $1/u_h$, $h=2,4,\LD ,l-1$, and $1/u_{l+h}=1/u_h$, $h=1,3,\LD,l$, appear with the opposite sign.
If $1/u_h$ and $1/u_k$ have the same sign, one can argue as in the case of the cluster points $V_i$, since $u_h,u_k$
are both $>1$. So we are left with the case
\BD
 1/u_h=-1/u_k+a, a\in\Z.
\ED
Since $0<1/u_h,1/u_k<1$, this is only possible with $a=1$. Accordingly, we obtain
\BD
  u_h=(1-1/u_k)^{-1}=1+1/(u_k-1).
\ED
Now $u_k=\lfloor\OV{ b_k,b_{k-1},\LD,b_1,b_l,b_{l-1},\LD,b_{k+1}}\rfloor.$
If $b_k>1$, we obtain
\BD
  1+1/(u_k-1)=\lfloor 1,b_k-1,\OV{b_{k-1},\LD,b_1,b_l,\LD ,b_k}\rfloor=u_h=
  \lfloor \OV{b_h,\LD, b_1,b_l,\LD,b_{h+1}}\rfloor.
\ED
However, $u_h$ is purely periodic with period length $l$, and so must be $1+1/(u_k-1)$. But if we compare the second
entry $b_k-1$ with the second entry that follows the period, we get the contradiction
$b_k-1=b_k$.
If $b_k=1$, we obtain
\BD
  1+1/(u_k-1)=\lfloor b_{k-1}+1,\OV{b_{k-2},\LD,b_1,b_l,\LD ,b_{k-1}}\rfloor.
\ED
Again, $1+1/(u_k-1)$ must be purely periodic with period length $l$, which gives the impossible
relation $b_{k-1}+1=b_{k-1}$.

If $l=L$ is even, we proceed in a similar way: We have the same sign in (\ref{1.1}) for $1/u_h$, $h=1,3,\LD,L-1$,
and the opposite sign for $1/u_h$, $h=2,4,\LD,L$ and, thus, can rule out the corresponding identities
if the signs are equal.
In the case $1/u_h=-1/u_k+a$, $a\in\Z$, we argue as above.
\end{proof}

\begin{proof}[Proof of Theorem \ref{t4}]
Let $h\in\{1,\LD, L\}$. We say that the \FL-convergent $p_k/q_k$ belongs to the \FL-{\em class} $h$, if
$k=q+nL+h$ for some $n\ge 0$. For the \FL-convergents $p_k/q_k$ belonging to this class, the corresponding
\DED sums $S(p_k,q_k)$ converge against the cluster point $U_h$ --- as we have shown in the paper \cite{Gi}.
In the same way we say that the \CE-convergent $s_j/t_j$
belongs to the \CE-{\em class} $i$, $i\in\{1,\LD, m\}$, if $j=r+nm+i$ for some $n\ge 0$. For \CE-convergents $s_j/t_j$
belonging to this class, the
corresponding \DED sums $S(s_j,t_j)$ converge against the cluster point $V_i$.

Let $i\in\{1,\LD,m\}$ be such that $d_{i+1}\ge 3$ (if $i=m$, $d_1\ge 3$). By Theorem \ref{t1}, each convergent
$s_j/t_j$ of the \CE-class $i$ equals some \FL-convergent $p_k/q_k$. Since $s_j/t_j>z$, $k$ must be odd.
Up to finitely many exceptions, these convergents $s_j/t_j$ must belong to exactly one \FL-class $h$ ---
otherwise the convergent sequence $S(s_j,t_j)$ would have more than one cluster point. Accordingly,
$S(s_j,t_j)$ converges against $U_h=V_i$ for this $h$. Because $q+h\equiv k$ mod 2 for all $p_k/q_k$ in the
\FL-class $h$, we see that $q+h$ is odd.

Finally, suppose that $q+h$ is even and $U_h=V_i$ for some $i\in\{1,\LD,m\}$. By (\ref{1.1}) and (\ref{1.3}),
$-1/u_h=-1/v_i+a$, $a\in\Z$.
As in the proof of Theorem \ref{t3}, we conclude $u_h=v_i$.
Now
\BD
  u_h=\lfloor\OV{b_h,b_{h-1},\LD,b_1,b_L,\LD b_{h+1}}\rfloor=
  \lceil{b_h+1,\OV{2^{(b_{h-1}-1)},b_{h-2}+2,\LD,2^{(b_{h+1}-1)}, b_h+2}}\rceil.
\ED
Since
$v_i=\lceil\OV{d_i,d_{i-1},\LD,d_1,d_m,\LD d_{i+1}}\rceil$ is purely periodic, we see that
$b_h+1$ must be equal to $b_h+2$, which is impossible.
\end{proof}


\end{document}